\documentclass[a4paper,10pt]{article}
\usepackage{amsmath}
\usepackage{amsfonts}
\usepackage{amssymb}
\usepackage{amsthm}

\usepackage{xspace}
\usepackage{calc}

\theoremstyle{plain}
\newtheorem{theorem}{Theorem}
\newtheorem{lemma}[theorem]{Lemma}

\newtheorem{question}[theorem]{Question}

\newcommand{\lem}[1]{Lemma~\ref{#1}}
\newcommand{\thm}[1]{Theorem~\ref{#1}}

\newcommand{\qn}[1]{Question~\ref{#1}}

\newcommand{\sect}[1]{Section~\ref{#1}}

\newcommand{\defn}{\emph}
\newcommand{\N}{\mathbb{N}}
\newcommand{\Z}{\mathbb{Z}}
\newcommand{\Q}{\mathbb{Q}}
\newcommand{\R}{\mathbb{R}}
\newcommand{\dyadics}{\mathbb{D}}
\newcommand{\sys}{P}
\newcommand{\column}[1]{c^{(#1)}}

\newcommand{\symmetric}{System~$\mathcal{A}$\xspace}
\newcommand{\QvN}{System~$\mathcal{B}$\xspace}
\newcommand{\asymmetric}{System~$\mathcal{C}$\xspace}

\newcommand{\SystemI}{System~$\mathcal{I}$\xspace}

\newtheorem*{symmetricDef}{System $\mathcal{A}$}
\newtheorem*{QvNDef}{System $\mathcal{B}$}
\newtheorem*{asymmetricDef}{System $\mathcal{C}$}

\title{Partition regularity in the rationals}

\author{Ben Barber
        \footnote{Department of Pure Mathematics and Mathematical Statistics,
                  Centre for Mathematical Sciences,
                  Wilberforce Road, Cambridge, CB3 0WB, UK.
                  {\tt b.a.barber@dpmms.cam.ac.uk}}
        \and
        Neil Hindman
        \footnote{Department of Mathematics,
                 Howard University,
                  Washington, DC 20059, USA.
                  {\tt nhindman@aol.com}}
        \thanks{This author acknowledges support received from the National
Science Foundation (USA) via Grant DMS-1160566.}
        \and
        Imre Leader
        \footnote{Department of Pure Mathematics and Mathematical Statistics,
                  Centre for Mathematical Sciences,
                  Wilberforce Road, Cambridge, CB3 0WB, UK.
                  {\tt i.leader@dpmms.cam.ac.uk}}
                  }

\begin{document}

\maketitle

\begin{abstract}
A system of homogeneous linear equations with integer coefficients is
\defn{partition regular} if, 
whenever the natural numbers are finitely coloured,
the system has a monochromatic solution.  The Finite Sums theorem provided the
first 
example of an infinite partition regular system of equations.  Since then,
other such systems of equations have been found, but each can be viewed as a
modification of the Finite Sums theorem.  

We present here a new infinite
partition regular system of equations that appears to arise in a genuinely
different way. This is the first example of a partition regular system in which
a variable occurs with unbounded coefficients. A modification of the system
provides an example of a system that is partition regular over $\Q$ but not
$\N$, settling another open problem. 
\end{abstract}

\section{Introduction}
An \defn{$r$-colouring} of the natural numbers is a partition of $\N = \{1, 2, 3, \ldots\}$ into $r$ parts $A_1, \ldots, A_r$.  Each $A_i$ is a \defn{colour class}, and a set $S \subseteq \N$ is \defn{monochromatic} if $S \subseteq A_i$ for some $i$.

A system of homogeneous linear equations with rational coefficients is \defn{partition regular} over $\N$ if, whenever the natural numbers are finitely coloured, there is some colour class in which the system has a solution.  For example, Schur's theorem states that, whenever the natural numbers are finitely coloured, there is some colour class containing numbers $x$, $y$ and $x+y$; that is, the single equation $x+y=z$ is partition regular.  Partition regularity over $\Z$ or $\Q$ or $\R$ is defined similarly, with the additional restriction that the values taken by the variables must be non-zero.

In 1933 Rado \cite{Rado} characterised finite partition regular systems in terms of a simple property of their matrix of coefficients.  Let $A$ be an $m \times n$ matrix with rational entries and let $\column{1}, \dotsc, \column{n}$ be the columns of $A$.  
We say that $A$ has the \defn{columns property} if there is a partition 
$[n] = I_1 \cup I_2 \cup \dotsb \cup I_s$ of the columns of $A$ such that 
$\sum_{i \in I_1} \column{i} = 0$, and, for each $t$,
\begin{equation*}
 \sum_{i \in I_t} \column{i} \in \langle \column{i} : i \in I_1 \cup \dotsb \cup I_{t-1}\rangle,
\end{equation*}
where $\langle \cdot \rangle$ denotes (rational) linear span and $[n]=\{1,2,\ldots,n\}$.
\begin{theorem}[Rado \cite{Rado}]
 Let $A$ be an $m \times n$ matrix with rational entries.  The system of linear equations $Ax = 0$ is partition regular over $\N$ if and only if $A$ has the columns property.
\end{theorem}

For infinite systems of equations much less is known.  The first example of an infinite partition regular system came from the Finite Sums theorem.

\begin{theorem}[Hindman \cite{Hindman}]\label{finite-sums}
 Whenever the natural numbers are coloured with finitely many colours, there exist $x_1 < x_2 < \cdots$ such that all finite sums $\sum_{i \in I} x_i$ where $I \neq \emptyset$ are the same colour.
\end{theorem}

(Formally, the partition regular system associated with \thm{finite-sums} consists of all equations $\sum_{i \in I} x_i = y_I$, where $\{y_I\}$ is a new set of variables indexed by the non-empty finite subsets of $\N$.)

\thm{finite-sums} has been generalised in various ways.  Two examples follow.

\begin{theorem}[Milliken \cite{Milliken} and Taylor \cite{Taylor}]\label{milliken-taylor}
 Whenever the natural numbers are finitely coloured there exist $x_1 < x_2 < \cdots$ such that all finite sums $\sum_{i \in I} x_i + \sum_{j \in J} 2x_j$, where $I, J \neq \emptyset$ and $\max I < \min J$, are the same colour.
\end{theorem}
This is the `$1,2$' version of Milliken and Taylor's result; there are corresponding versions for any finite string of natural numbers.

\begin{theorem}[Deuber-Hindman \cite{Deuber-Hindman}]\label{deuber-hindman}
For any sequence $\sys_1, \sys_2, \ldots$ of finite partition regular systems of equations, whenever the natural numbers are finitely coloured there is a sequence of corresponding solution sets $S_1, S_2, \ldots$ such that all finite sums of the form $\sum_{i \in I} x_i$, where $I \neq \emptyset$ and $x_i \in S_i$ for
all $i \in I$, are the same colour.
\end{theorem}

Here a \defn{solution set} for a system $\sys$ is the set of values taken by the variables in some 
solution to $\sys$.  For example, the $\sys_i$ could be chosen such that each $S_i$ contains
an arithmetic progression of length $i$. 

There is no known classification of infinite partition regular systems: in fact, Theorems \ref{finite-sums}--\ref{deuber-hindman}, together with some
other systems in \cite{central}, are almost everything that is known. 

The Milliken-Taylor theorem can be proved by mimicking the proof of Ramsey's theorem, replacing appeals to the pigeonhole principle with appeals to the Finite Sums theorem.  The Deuber-Hindman theorem can be proved by modifying the ultrafilter proof (due to Galvin and Glazer---see \cite{Comfort}) of the Finite Sums theorem.  All known infinite partition regular systems until now have been strongly related to the Finite Sums theorem: it is therefore of interest to determine whether all infinite partition regular systems must be related to the Finite Sums theorem in some way.

The systems associated with Theorems \ref{finite-sums}--\ref{deuber-hindman} 
(and also the systems in \cite{central}) each have the property that every variable appears with only a bounded set of coefficients (although the set of all coefficients can be unbounded in the case of the Deuber-Hindman theorem).  This led to the following question in Hindman, Leader and Strauss's survey of open problems in partition regularity \cite{problem-survey}.
\begin{question}[\cite{problem-survey}] \label{q:boundedness}
 Are there infinite partition regular systems in which some variable appears with an unbounded set of coefficients?
\end{question}

A small amount of care is clearly needed in interpreting this question to avoid  trivialities. For example, we do not want to allow a system formed by taking many integer multiples of one fixed equation (e.g.~the system $x+y=z$, $2x+2y=2z$, \ldots). The obvious way to formalise this is to ask that there is some variable $x$ such that the ratio of the coefficient of $x$ to the sum of the other coefficients in each equation is unbounded.%
\footnote{We mention in passing that one cannot formalise the question by an alternative approach of the form `cannot be rewritten such that \ldots'. This is because {\it any} system of equations can be rewritten so that every variable appears in only finitely many equations.  (This is an easy exercise in Gaussian elimination.)}

The first result of this paper is that there {\it are} such systems. This is perhaps somewhat surprising. Indeed, for just about any system of equations that one writes down with a variable having unbounded coefficients, it is usually possible to find a colouring of $\N$ without a monochromatic solution. See \cite{forb} for several examples of this.

Roughly speaking, our idea is to try to have a variable ($y$ below) that will appear, with unbounded coefficients, as the difference between two different expressions (that do not involve $y$). Now, it is well known that a system of equations such as $y=x_1-z_1, \ 2y=x_2-z_2, \ 3y=x_3 -z_3, \ldots$ is not partition regular (see for example \cite{forb} for this and some considerable strengthenings). So we seek to replace the $x_i$ and $z_i$ by some expressions with more `flexibility'. One way to do this is to allow some sums instead.  Here is a system that arises in this way.

\begin{symmetricDef}
 \begin{align*}
                      x_{1,1} + 2y & = z_{1,1} \\
            x_{2,1} + x_{2,2} + 4y & = z_{2,1} + z_{2,2} \\
                                   & \mathrel{\makebox[\widthof{=}]{\vdots}} \\
x_{n,1} + \cdots + x_{n,n} + 2^n y & = z_{n,1} + \cdots + z_{n,n}\\
                                   & \mathrel{\makebox[\widthof{=}]{\vdots}}
 \end{align*}
\end{symmetricDef}

In \sect{sec:symmetric} we show that \symmetric is partition regular.
In fact, our proof gives more than this: we show that \symmetric remains partition regular if the coefficients $2^n$ are replaced by any other sequence of integer coefficients.
Curiously, our argument involves density considerations, even though there is
no `density version' of this statement---for example, the set of numbers that
are congruent to 1 modulo 3 does not contain any solution to \symmetric.

Let us now turn to a consideration of partition regularity over different
spaces. A system of equations which is 
partition regular over $\N$ is clearly partition regular over $\Z$, and it is
easy to see that the converse holds as well. (Given a bad colouring of
$\N$ with $k$ colours, we obtain a bad colouring of $\Z$ with $2k$ colours 
by `reflecting' this
colouring to the negative numbers using $k$ new colours.) 

If a system is
partition regular over $\N$ then it is certainly partition regular over $\Q$,
and similarly if it is partition regular over $\Q$ then it is partition
regular over $\R$. What about the converses? It is known 
(see \cite{problem-survey}) that the reals are `richer' than the rationals:
there are systems that are partition regular over $\R$ but not over $\Q$.
But the question of whether or not $\N$ and $\Q$ can be distinguished remained
open:

\begin{question}[\cite{problem-survey}] \label{q:QvN}
 Are there infinite systems that are partition regular over $\Q$ but not 
over $\N$?
\end{question}

Here we show that a modification to \symmetric provides an example of such
a system. Again, this is rather unexpected: for all known examples of systems 
not partition regular over $\N$ there were easy extensions of the bad 
colourings
to bad colourings of $\Q$, using `factorial base' and similar ideas (see for
example \cite{forb}
and \cite{problem-survey}). So the general belief was that the answer to
\qn{q:QvN} was probably negative.

\begin{QvNDef}
 \begin{align*}
                    x_{1,1} + 2^{-1}y & = z_{1,1} \\
          x_{2,1} + x_{2,2} + 2^{-2}y & = z_{2,1} + z_{2,2} \\
                                      & \mathrel{\makebox[\widthof{=}]{\vdots}} \\
x_{n,1} + \cdots + x_{n,n} + 2^{-n} y & =z_{n,1} + \cdots + z_{n,n}\\
                                      & \mathrel{\makebox[\widthof{=}]{\vdots}}
 \end{align*}
\end{QvNDef}

In \sect{sec:QvN} we show that \QvN is partition regular over $\Q$.  Since \QvN does not even have any solutions over $\N$, this answers \qn{q:QvN}.

The argument for \symmetric uses the symmetry of the system (that there are equal
numbers of $x$'s and $z$'s on each side of the equations). In \sect{sec:asymmetric} we 
modify the argument to show that the following less symmetric system is also partition regular.

\begin{asymmetricDef}
 \begin{align*}
                      x_{1,1} + 2y & = z_1 \\
            x_{2,1} + x_{2,2} + 4y & = z_2 \\
                                   & \mathrel{\makebox[\widthof{=}]{\vdots}} \\
x_{n,1} + \cdots + x_{n,n} + 2^n y & = z_n\\
                                   & \mathrel{\makebox[\widthof{=}]{\vdots}}
 \end{align*}
\end{asymmetricDef}

This fact solves a problem in image partition regularity from \cite{DH}, namely
`if a system is image partition regular over $\N$, does it follow that we
can always find a monochromatic solution in the reals with all values as small as we
please?' We give the details in \sect{sec:asymmetric}.


\section{\symmetric}\label{sec:symmetric}

\symmetric has a solution in colour class $A$ if and only if there is some $y \in A$ with
 \[
   2^n y \in nA-nA
 \]
for every $n \in \N$, where we use the standard notation for sumsets and difference sets
\begin{align*}
 A + B & = \{a + b : a \in A, b \in B\} \\
 A - B & = \{a - b : a \in A, b \in B\} \\
      kA & = \underbrace{A + \cdots + A}_{k \text{ times}}.
\end{align*}

We expect iterated sumsets and difference sets to have some additive structure, so this trivial rewriting of \symmetric suggests that a sensible first question might be `what kind of structure can we find inside $nA-nA$ when $n$ is large?'

We shall need a notion of density.  For a subset $S$ of $\N$, its (\defn{upper}) \defn{density} is
\[
 d(S) = \limsup_{n \to \infty} \frac{|S \cap [n]|}{n}.
\]
For a subset $S$ of $\Z$ we write $d(S) = d(S \cap \N)$ (that is, we measure its density `in one direction').  We call $S$ \defn{dense} if $d(S) > 0$.  Density has the following properties:
\begin{itemize}
 \item If $A \subseteq B$, then $d(A) \leq d(B)$.
 \item For any $A$ and $B$, $d(A\cup B) \leq d(A) + d(B)$.  In particular, whenever $\N$ is finitely coloured, at least one colour class is dense.
 \item If $A + x_1, \ldots, A+x_k$ are disjoint translates of a set $A$, then
\[
 d\left(\,\bigcup_{i=1}^k \,(A+ x_i)\right) = kd(A).
\]
  This is because $A$ and $A+x_i$ are dense in roughly the same intervals.  (It is not true in general that if $A$ and $B$ are disjoint then $d(A \cup B) = d(A) + d(B)$: indeed, it is easy to construct infinitely many pairwise disjoint sets that each have density $1$ by having them `take turns' to have density close to $1$ on initial segments of $\N$.)
\end{itemize}

For dense sets things work as well as we could hope for.  Write $m \cdot S=\{ms:s\in S\}$ for the set obtained from $S$ by pointwise multiplication by $m$.  The following is a slight generalisation of a result of Stewart and Tijdeman \cite{Stewart-Tijdeman}, whose argument covers the case where $k$ is a power of $2$.

\begin{lemma}\label{eventually-symmetric}
 Let $S$ be a dense, symmetric subset of $\Z$ containing 0.  Then there is 
an $m \in\N$ such that, for $k\geq 2/d(S)$, $kS = m\cdot\Z$.

In particular, if $A$ is a dense subset of $\N$, then there is an $m \in\N$ such that, for
$k\geq 2/d(A)$, $kA - kA = m\cdot\Z$.
\end{lemma}

\begin{proof}  We will show that there is some $j \leq 1/d(S)$ such that
$(2j+1)S=(2j)S$, and so $(k+1)S=kS$ for $k\geq 2 / d(S)$.  Once we have shown this, let $k = \lceil 2/d(S) \rceil$ and let $X = kS$.
We have that $X + X = kS + kS = (2k)S = X$.  Since $S$ is symmetric, we also have that $X=-X$, 
so $X$ is closed under addition and the taking of inverses, 
hence is a subgroup of $\Z$ as required.

Suppose instead that for each $j \leq 1/d(S)$,
$(2j)S\subsetneq (2j+1)S$. We claim that, for each such $j$,
$(2j+1)S$ contains $j+1$ disjoint translates of $S$.  
This is a contradiction for $j=\lfloor1/d(S)\rfloor$.

The claim is true for $j=0$.  For $j>0$, choose $x \in (2j+1)S \setminus (2j)S$.
Then $x=s_1 + \cdots + s_{2j+1}$ with $s_i\in S$ for each $i$.  We have
\begin{align*}
 S + x - s_1 = S + s_2 + \cdots + s_{2j+1} & \subseteq (2j+1)S \\
 \hbox{and }(2j-1)S - s_1 \subseteq (2j)S & \subseteq (2j+1)S.
\end{align*}
Hence it suffices to show that $S + x - s_1$ and $(2j-1)S - s_1$ are disjoint.  
But if they intersect then $t_0 + x - s_1 = t_1 + \cdots + t_{2j-1} - s_1$ for some $t_i \in S$, 
from which it follows that $x = t_1 + \cdots + t_{2j-1} - t_0 \in (2j)S$, contradicting the choice of $x$.

For the `in particular' statement, note that $d(A-A)\geq d(A)$.
\end{proof}

Given a colour class $A$, let $m$ be as above.  If there is a $y \in A$ with $y$ divisible by $m$ then \lem{eventually-symmetric} tells us that we can solve \symmetric `eventually' inside $A$.  That leaves only finitely many equations unsolved; we hope to solve these using Rado's theorem.  The problem is that the colour class obtained from Rado's theorem might not be dense.  To avoid this situation we work inside a long homogeneous arithmetic progression that is disjoint from the non-dense colour classes.

\begin{lemma} \label{long-homogeneous-AP}
 Let $\N = A \cup B$ where $d(B) = 0$.  Then, for any $l\in\N$, there is a $d \in \N$ such that $d \cdot [l] \subseteq A$.
\end{lemma}

\begin{proof}
Since $B$ has density 0 there is an $n \geq l$ such that 
$|B \cap [n]| < \displaystyle\frac{n}{2l^2}$.  Let 
$D = \lfloor n/l \rfloor \geq\displaystyle \frac{n}{2l}$ and let $E = \bigcup_{d=1}^D (d \cdot [l]) \subseteq [n]$.  We claim that $d \cdot [l] \subseteq A$ for some $d \leq D$.  Indeed, suppose not.  Then, for every $d\leq D$, $d \cdot [l]$ contains an element of $B$.  Each element of $E$ is in at most $l$ of the sets $d \cdot [l]$, so
\[
 |B \cap [n]| \geq |B \cap E| \geq \frac{D}{l} \geq \frac{n}{2l^2},
\]
contradicting the choice of $n$.
\end{proof}

We can now show that \symmetric is partition regular.

\begin{theorem} \label{symmetric-case}
 Let $\N = A_1 \cup \dotsb \cup A_r$ be an $r$-colouring of $\N$.  Then there is a colour class $A_i$ containing a solution to \symmetric.
\end{theorem}

\begin{proof}
 For each dense colour class $A_i$, apply \lem{eventually-symmetric} to obtain $m_i$ and $K_i$ such that, for $k \geq K_i$, $kA_i - kA_i = m_i \cdot \Z$.  Let $m$ be the least common multiple of the $m_i$, and let $K$ be the maximum of the $K_i$.  For every dense colour class $A_i$, and $k \geq K$, we have $kA_i - kA_i \supseteq m \cdot \Z$.
 
Write $P$ for the system consisting of the first $K-1$ equations of \symmetric.  It is easy to check that $P$ satisfies the columns property, so by Rado's theorem it is partition regular.  It follows that there exists an $l$ such that whenever $[l]$, or any progression $c \cdot [l]$, is $r$-coloured it contains a monochromatic solution to $P$.

Apply \lem{long-homogeneous-AP} to get $d$ with $d \cdot [ml]$ disjoint from the non-dense colour classes.  Then $md \cdot [l] \subseteq d \cdot [ml]$ is also disjoint from the non-dense colour classes, and by the choice of $l$ there is a dense colour class $A_i$ such that $A_i \cap \left(md \cdot [l]\right)$ contains a solution
\[
 \left\{y, x_{1,1}, \dotsc, z_{K-1,K-1}\right\}
\]
to $P$, where every element of the solution set is divisible by $m$. Then, for $k \geq K$, $2^k y \in kA_i - kA_i$, so this solution to $P$ can be extended to a solution for \symmetric inside $A_i$.
\end{proof}

Our proof used only that the coefficients $2^n$ were integers to find a solution to \symmetric inside some dense colour class.  In the next section we show that we can allow rational coefficients if we colour $\Q$ instead of $\N$.


\section{\QvN} \label{sec:QvN}

\QvN is not partition regular over $\N$ because it has no solutions over $\N$.  In this section we show that \QvN is partition regular over $\Q$.

Given a finite colouring of $\Q$ we seek a colour class $A$ and a $y \in A$ such that
 \[
   2^{-n} y \in nA-nA
 \]
for every $n \in \N$.  The idea is to view $\Q$ as infinitely many nested copies of $\Z$ and apply the methods of the previous section.

It is convenient to work inside the dyadic rationals $\dyadics = \bigcup_{j=0}^\infty (2^{-j} \cdot \Z)$.  We call the subset $2^{-j} \cdot \Z$ the $j$th \defn{level} of $\dyadics$.  For a set $S \subseteq \dyadics$, let
\[
 d_j(S) = \limsup_{n \to \infty} \frac{|S \cap (2^{-j} \cdot[n])|}{n}
\]
be the density of $S$ in the $j$th level, and let
\[
 d^*(S) = \limsup_{j \to \infty} d_j(S).
\]
If $\dyadics = A_1 \cup \dotsb \cup A_r$, then, for each $j$,
\[
 1 \leq d_j(A_1) + \dotsb + d_j(A_r),
\]
so
\[
 1 \leq d^*(A_1) + \dotsb + d^*(A_r),
\]
and $d^*(A_i) > 0$ for at least one $i$.

The proof follows the same pattern as before.

\begin{lemma} \label{eventually-Q}
 Let $A \subseteq \dyadics$ with $d^*(A) > 0$.  Then there is an $m \in\N$ such that, for 
$k\geq 4/d^*(A)$, $kA - kA \supseteq m\cdot\dyadics$.
\end{lemma}

\begin{proof}
 There are infinitely many levels $j$ such that $d_j(A) > d^*(A) / 2$.  
Let $j$ be one such level and let $A_j=A\cap 2^{-j}\cdot\Z$.  Then by
\lem{eventually-symmetric} inside level $j$, there is an 
$m_j \in \N$ such that, for $k \geq 4/d^*(A) > 2 / d_j(A)=2/d_j(A_j)$, 
\[
kA-kA \supseteq kA_j-kA_j= 2^{-j} m_j \cdot \Z\,.
\]
Then $\frac{1}{m_j}=d_j(2^{-j} m_j \cdot \Z)=d_j(kA_j-kA_j)\geq d_j(A_j)=d_j(A)$,
so $m_j \leq 1/d_j(A) \leq 2 / d^*(A)$. Thus some $m_j$ occurs infinitely often: call it $m$.  Choose an infinite set $J$ of levels  such that, for $k \geq 4 / d^*(A)$, $kA-kA \supseteq 2^{-j} m \cdot \Z$.  Then, for $k \geq 4 / d^*(A)$,
\begin{align*}
  kA-kA & \supseteq \bigcup_{j \in J} (2^{-j} m \cdot \Z) \\
        & = \bigcup_{j=0}^\infty (2^{-j} m \cdot \Z) \\
        & = m \cdot \dyadics,
\end{align*}
since the levels of $\dyadics$ are nested and $J$ is infinite.
\end{proof}

\begin{lemma} \label{long-homogeneous-AP-Q}
 Let $\dyadics = A \cup B$ where $d^*(B) = 0$.  Then, for any $l\in\N$, 
there exist $j$ and $d$ such that $2^{-j}d \cdot [l] \subseteq A$.
\end{lemma}

\begin{proof}
Since $d^*(B) = 0$ there is a $j$ such that $d_j(B) < \displaystyle\frac{1}{4l^2}$.  
Then there is an $n \geq l$ such that $|B \cap (2^{-j}\cdot[n])| < \displaystyle\frac{n}{2l^2}$.  
The rest of the proof is the same as that of \lem{long-homogeneous-AP} with
$E=\bigcup_{d=1}^D 2^{-j}d\cdot[l]$.
\end{proof}

\begin{theorem} \label{Q}
 Let $\dyadics = A_1 \cup \dotsb \cup A_r$ be an $r$-colouring of $\dyadics$.  Then there is a colour class $A_i$ containing a solution to \QvN.
\end{theorem}

\begin{proof}
 For each dense colour class $A_i$, apply \lem{eventually-Q} to obtain $m_i$ and $K_i$ such that, for $k \geq K_i$, $kA_i - kA_i \supseteq m_i \cdot \dyadics$.  Let $m$ be the least common multiple of the $m_i$, and let $K$ be the maximum of the $K_i$.  For every dense colour class $A_i$, and $k \geq K$, we have $kA_i - kA_i \supseteq m \cdot \dyadics$.
 
Write $P$ for the system consisting of the first $K-1$ equations of \QvN.  It is easy to check that $P$ satisfies the columns property, so by Rado's theorem it is partition regular.  It follows that there exists an $l$ such that whenever $[l]$, or any progression $2^{-j}c \cdot [l]$, is $r$-coloured it contains a monochromatic solution to $P$.

Apply \lem{long-homogeneous-AP-Q} to get $j$ and $d$ with $2^{-j}d \cdot [ml]$ disjoint from the colour classes with $d^*(A_i) = 0$.  Then $2^{-j}md \cdot [l] \subseteq 2^{-j}d \cdot [ml]$ is also disjoint from the non-dense colour classes, and by the choice of $l$ there is a dense colour class $A_i$ such that $A_i \cap \left(2^{-j}md \cdot [l]\right)$ contains a solution
\[
 \left\{y, x_{1,1}, \dotsc, z_{K-1,K-1}\right\}
\]
to $P$, where every element of the solution set is in $m \cdot \dyadics$.  Then, for $k \geq K$, $2^{-k} y \in kA_i - kA_i$, so this solution to $P$ can be extended to a solution for \QvN inside $A_i$.
\end{proof}


\section{\asymmetric}\label{sec:asymmetric}

We now turn our attention to \asymmetric.  Given a finite colouring of $\N$ we seek a colour class $A$ and a $y \in A$ for which
 \[
   2^n y \in A-nA
 \]
for every $n \in \N$.  The next two lemmas generalise \lem{eventually-symmetric} to the asymmetric case.

\begin{lemma}\label{translated}
 Let $S$ be a dense subset of $\Z$ with $0 \in S$.  Then there is an $X \subseteq \Z$ such that, 
for any $k\geq 2/d(S)$, we have
$S-kS = X$.
\end{lemma}

\begin{proof}
 As in the proof of \lem{eventually-symmetric}, 
we suppose to the contrary that $S-(2j)S \subsetneq S-(2j+1)S$
for all $j \leq 1/d(S)$ and show that $S-(2j+1)S$ contains
$j+1$ disjoint translates of $S$ for each such  $j$, which is impossible for $j=\lfloor\frac{1}{d(S)}\rfloor$.

The claim is true for $j=0$.  For $j>0$, choose $x \in \big(S-(2j+1)S\big) \setminus \big(S-(2j)S\big)$.  
Then $x=s_0 - s_1 - \dotsb - s_{2j+1}$ with $s_i \in S$ for each $i$.  We have
\begin{align*}
 S + x - s_0 & \subseteq S - (2j+1)S, \hbox{ and}\\
 S - (2j-1)S - s_0 & \subseteq  S - (2j)S \subseteq S - (2j+1)S,
\end{align*}
hence it suffices to show that $S + x - s_0$ and $S - (2j-1)S - s_0$ are disjoint.
But if they intersect then $t_0 + x - s_0 = t_1 - t_2 - \dotsb - t_{2j} - s_0$ for
some $t_i\in S$, whence $x = t_1 - t_2 - \cdots - t_{2j} - t_0 \in S - (2j)S$, contradicting the choice of $x$.
\end{proof}
What can we say about $X$?
\begin{lemma}\label{cosets}
 Let $A$ be a dense subset of $\N$.  Then there is an $m$ such that, for $k\geq 2/d(A)$,
$A-kA$ is a union of cosets of $m \cdot \Z$.
\end{lemma}
\begin{proof}
Let $k \geq 2/d(A)$, and
let $X=A - kA$. For any $a \in A$, we have by \lem{translated} that
\[ 
(A-a) - k(A-a) = (A-a)-(k+1)(A-a),
\]
and so 
\[
X=X-A+a.
\]
Since $a \in A$ was arbitrary it follows that
$X=X+A-A$, whence $X=X+l(A-A)$ for all $l$. Taking $l$ sufficiently large, by \lem{eventually-symmetric} there is an $m \in \Z$ such that $X = X + m \cdot \Z$.  Thus $X$ is a union of cosets of $m \cdot \Z$.
\end{proof}

The example of the odd numbers shows that it is not necessarily the case that $A-kA \supseteq m\cdot\Z$ for
large $k$.  However, the obstruction is clear: in that case $A$ does not contain even a single 
element of $m \cdot \Z$.  But then we can pass down to a copy of $\N$ coloured with one fewer 
colour and apply induction.  Combining this idea with the argument for \symmetric concludes 
the proof that \asymmetric is partition regular.

\begin{theorem} \label{asymmetric-case}
 Let $\N = A_1 \cup \dotsb \cup A_r$ be an $r$-colouring of $\N$.  Then there is a colour class $A_i$ containing a solution to \asymmetric.
\end{theorem}

\begin{proof}
 Suppose first that there is an $m$ and an $i$ such that $A_i$ is disjoint from $m\cdot\Z$.  Then $m \cdot \N$ is $(r-1)$-coloured by the remaining colour classes, so by induction on $r$ we can find a monochromatic solution to \asymmetric 
inside $m \cdot \N$.

Otherwise we may assume that every colour class meets every subgroup of $\Z$.
Apply \lem{cosets} to each dense colour class $A_i$ to obtain $m_i$ and $K_i$ such that,
for $k \geq K_i$, $A_i - kA_i$ is a union of cosets of $m_i \cdot \Z$.
Let $m$ be the least common multiple of the $m_i$, and let $K$ be the maximum of the $K_i$.  Then, for every dense colour class $A_i$, and 
$k \geq K$, $A_i - kA_i$ is a union of cosets of $m \cdot \Z$.  Since every colour class contains a multiple of $m$, one of those cosets is $m \cdot \Z$ itself.

 Write $P$ for the system consisting of the first $K-1$ equations of \asymmetric.  Again, it is easy to check that $P$ satisfies the columns property, so by Rado's theorem it is partition regular.  It follows that there exists an $l$ such that whenever $[l]$, or a progression $c \cdot [l]$, is $r$-coloured it contains a monochromatic solution to $P$.

Apply \lem{long-homogeneous-AP} to get $d$ with $d \cdot [ml]$ disjoint from the non-dense colour classes.  Then $md \cdot [l] \subseteq d \cdot [ml]$ is also disjoint from the non-dense colour classes, and by the choice of $l$ there is a dense colour class $A_i$ such that $A_i \cap \left(md \cdot [l]\right)$ contains a solution
\[
 \left\{y, x_{1,1}, \dotsc, z_{K-1}\right\}
\]
to $P$, where every element of the solution set is divisible by $m$. Then, for $k \geq K$, $2^k y \in A_i - kA_i$, so this solution to $P$ can be extended to a solution for \asymmetric inside $A_i$.
\end{proof}

 Noting that a dense subset of $m \cdot \N$ is dense in $\N$, we see that the proof actually shows that there is a solution inside a dense colour class.  Note also that, as for \thm{symmetric-case}, the proof of \thm{asymmetric-case} works if the sequence $(2^n)$ is replaced by any other integer sequence.


Until now we have considered \defn{kernel partition regularity}.  That is, given a colouring of $\N$ we have sought monochromatic kernel vectors (solutions to $Ax = 0$).  There is a corresponding notion of \defn{image partition regularity}: given a colouring of $\N$ and a matrix $A$ with rational entries (and only finitely many non-zero entries in each row), can we find an $x$ such that the image vector $Ax$ is monochromatic?  (Note that we do \emph{not} care about the colours of the entries of $x$.)  For example, the length 4 version of van der Waerden's theorem is the assertion
that the following matrix is image partition regular.
\[
\begin{pmatrix}
 1 & 0 \\
 1 & 1 \\
 1 & 2 \\
 1 & 3
\end{pmatrix}
\]

Just as kernel partition regularity can be thought of as a statement about monochromatic solutions to sets of \emph{equations}, image partition regularity can be thought of as a statement about monochromatic sets of \emph{expressions}.  For example, consider the set of expressions
\[ \renewcommand*{\arraystretch}{1.5}
 \setlength{\arraycolsep}{2em}
 \begin{array}{rcc}
                     x_{1,1} + 2y  & x_{1,1} & y \\
           x_{2,1} + x_{2,2} + 4y  & x_{2,1} & \\
 x_{3,1} + x_{3,2} + x_{3,3} + 8y  & x_{2,2} & \\
 \mathrel{\makebox[\widthof{$8y + 8y$}]{\vdots}}                            & \vdots &
 \end{array}
\]
consisting of the left hand sides of the equations of \asymmetric and the variables they contain.  This set of expressions is image partition regular because for any values of the variables the expressions form a solution set for \asymmetric.  Call this image partition regular system \SystemI.

We call a matrix $A$ \defn{image partition regular over $\R$ near zero} if, for all $\delta > 0$, whenever $(-\delta, \delta)\setminus\{0\}$ is finitely coloured there is a vector $x$ with real entries such that $Ax$ is monochromatic and contained in $(-\delta, \delta)\setminus\{0\}$.  In \cite[Question 3.10]{DH}, De and Hindman asked whether there is a system which is image partition regular over $\N$ but is not image partition regular over $\R$ near zero.  
We now show 
that \SystemI is such a system. This is of particular interest because it
is the `final link' which establishes that a 
diagram of implications in \cite{DH} involving nineteen properties does 
not have any 
missing implications.  

\begin{theorem}\label{iprnz} \SystemI is image partition regular over $\N$ but not image partition regular over $\R$ near zero.
\end{theorem}

\begin{proof}
We have already seen that \SystemI is image partition regular over $\N$.  To see that it is not image partition regular over $\R$ near zero choose $\delta > 0$ and partition $(-\delta, \delta)\setminus\{0\}$ as $(-\delta, 0) \cup (0, \delta)$.  If \SystemI is to be monochromatic then $y$ and each of the $x_{i,j}$ must lie in the same part: without loss of generality they are all positive.  But then
\[
 x_{n,1} + \cdots + x_{n,n} + 2^n y > 2^n y > \delta
\]
for $n$ sufficiently large.
\end{proof}

\bibliographystyle{plain}

\end{document}